\documentclass[11pt,reqno]{amsart}
\usepackage{amsmath, amsfonts, amsthm, amssymb, mathtools}
\usepackage{geometry}
\theoremstyle{plain}
\newtheorem{theorem}{Theorem}[section]
\newtheorem{corollary}[theorem]{Corollary}

\newtheorem{lemma}[theorem]{Lemma}
\newtheorem{proposition}[theorem]{Proposition}
\theoremstyle{definition}

\theoremstyle{remark}
\newtheorem*{remark}{Remark}

\newtheorem*{example}{Example}
\numberwithin{equation}{section}

\newcommand{\R}{\mathbb R}
\newcommand{\N}{\mathbb N}
\newcommand{\Z}{\mathbb Z}
\newcommand{\C}{\mathbb C}

\newcommand{\FOne}{\overline{U}}
\newcommand{\FTwo}{\overline{U2}}
\newcommand{\FThree}{U2}
\newcommand{\fone}{\overline{u}}
\newcommand{\ftwo}{\overline{u2}}
\newcommand{\fthree}{u2}
\newcommand{\Arg}{\operatorname{Arg}}

\begin{document}
\allowdisplaybreaks

\subjclass[2010]{05A16, 11F03, 11P81, 11P82}
\keywords{unimodal sequences, strongly unimodal sequences, partitions, overpartitions,
unimodal ranks, partition ranks, 
Dyson rank, $M_2$-rank, asymptotics, modular forms, mock modular forms}

\begin{abstract}
In this paper we study generating functions resembling the rank of strongly
unimodal sequences. We give combinatorial interpretations, 
identities in terms of mock modular forms, asymptotics, and a parity result.
Our functions imitate a relation between the rank of strongly unimodal
sequences and the rank of integer partitions. 
\end{abstract}

\title{Unimodal sequence generating functions arising from partition ranks}
\author[K. Bringmann]{Kathrin Bringmann}
\author[C. Jennings-Shaffer]{Chris  Jennings-Shaffer}

\address{University of Cologne, Faculty of Mathematical and Natural Sciences, Mathematical Institute, Weyertal 86-90, 50931 Cologne, Germany}
\email{kbringma@math.uni-koeln.de}
\email{cjenning@math.uni-koeln.de}

\thanks{\hspace{-0.6cm} The research of the first author is supported by the Alfried Krupp Prize for Young University Teachers of the Krupp foundation and the research leading to these results receives funding from the European Research Council under the European Union's Seventh Framework Programme (FP/2007-2013) / ERC Grant agreement n. 335220 - AQSER}

\thispagestyle{empty} \vspace{.5cm}
\maketitle

\section{Introduction and statement of results}

A sequence $\{a_j\}_{j=1}^s$ of positive integers is a {\it unimodal 
sequence of size $n$} if it is of the form
\begin{equation*}
a_1\le a_2\le \dotsb \le a_k \ge a_{k+1} \ge \dotsb \ge a_{s}
\qquad\mbox{and}\qquad
a_1+ a_2+ \dotsb + a_{s} = n.
\end{equation*}
The maximum value, $a_k$, is called the {\it peak}.
If the inequalities are strict, the sequence is called {\it strongly unimodal}.
Such sequences are related to integer partitions. Recall that a finite 
sequence $\{a_j\}_{j=1}^s$ of positive integers is a {\it partition of 
size $n$} if it is of the form
\begin{equation*}
a_1\ge a_2\ge \dotsb \ge a_{s}
\qquad\mbox{and}\qquad
a_1+ a_2+ \dotsb + a_{s} = n.
\end{equation*}

Unimodal sequences, partitions, and similar objects appear
throughout modern and classical literature on a variety of subjects including
algebra, combinatorics, number theory, physics, and special functions
\cite{Andrews1, Andrews2, Stanley1, Stanley2}. To 
motivate our results, we discuss a few highlights in 
number theory. We focus on strongly unimodal sequences, rather than
unimodal sequences. The enumeration function for strongly unimodal sequences 
seems to have first appeared as $X_d(n)$ in \cite{Andrews1}, along with a 
wealth of related functions;
unimodal sequence type counting functions can also be found in older works
such as \cite{Andrews4, Auluck1, Wright1}.

We let $p(n)$ denote the number of partitions of size $n$
and let $u(n)$ denote the number of strongly unimodal sequences of size $n$. We note that the conventions for zero are somewhat inconsistent 
between strongly unimodal sequences and partitions, as we set
$u(0):=0$ but $p(0):=1$.
By standard counting 
techniques, the generating functions for partitions and strongly unimodal sequences are 
\begin{align*}
P(q) 
:= 
	\sum_{n\ge0} p(n)q^n 
	=
	\sum_{n\ge0} \frac{q^{n^2}}{(q;q)_{n}^2}
	=
	\frac{1}{(q;q)_\infty}		
, \quad 
U(q) 
:= 
	\sum_{n\ge1} u(n)q^n 
	=
	\sum_{n\ge1} (-q;q)_{n-1}^2 q^{n}
,
\end{align*}
where we use the {\it $q$-Pochhammer symbol}, 
$(a;q)_n:=\prod_{j=0}^{n-1}(1-aq^j)$ for $n\in\N_0\cup\{\infty\}$.
This extends to arbitrary $n$ by setting
$(a;q)_n: = \frac{(a;q)_\infty}{(aq^n;q)_\infty}$.
Furthermore, we let
$(a_1,a_2,\dots,a_k;q)_n := (a_1;q)_n (a_2;q)_n \dotsm (a_k;q)_n$.
Throughout the article, $q$ is a complex
variable with $0<|q|<1$.

Perhaps the three most famous results for the partition function
are the following. There is the asymptotic formula of Hardy and
Ramanujan \cite[equation (1.41)]{HardyRamanujan1}
\begin{align}\label{Eq:PartitionAsymptotic}
p(n) &\sim \frac{1}{4\sqrt{3n}}e^{\pi\sqrt{\frac{2n}{3}}}
,\quad \mbox{as }n\rightarrow\infty.
\end{align}
Moreover the function 
\begin{align*}
\eta(\tau)
&:=
q^{\frac{1}{24}}\prod_{n\ge1}\left(1-q^n\right)
	=
\frac{q^{\frac{1}{24}}}{P(q)}
,\qquad
(q:=e^{2\pi i\tau} \mbox{ throughout}),
\end{align*}
is Dedekind's eta function, which is a modular form of weight $\frac12$ (with multiplier).
Lastly, there are Ramanujan's congruences \cite{Ramanujan1} 
\begin{align*}
p(5n+4)\equiv0\pmod{5},\qquad
p(7n+5)\equiv0\pmod{7},\qquad
p(11n+6)\equiv0\pmod{11}.
\end{align*}

By \cite[Corollary 1.2]{Rhoades1}, $u(n)$ has a similar asymptotic behavior to 
\eqref{Eq:PartitionAsymptotic}
\begin{align*}
u(n) &\sim \frac{1}{8\cdot 6^{\frac14}n^{\frac34}}e^{\pi\sqrt{\frac{2n}{3}}}
,\quad \mbox{as }n\rightarrow\infty.
\end{align*}
However, $U(q)$ is essentially a mixed mock modular form instead of a modular form.
A {\it mock modular form} 
is the holomorphic part of a harmonic Maass form 
with nontrivial non-holomorphic part.
A {\it harmonic Maass form} is a function that 
transforms like a modular form, satisfies similar growth conditions, 
but needs to only be smooth and annihilated by the weighted 
hyperbolic Laplacian.
A {\it mixed mock modular form} is basically an element of the tensor space of
modular forms and mock modular forms.
These terms and their encompassing theory can be found in \cite{BringmannFolsomOnoRolen1}.
While $u(n)$ does not satisfy congruences as elegant as those of $p(n)$,
it turns out that
\begin{align*}
u\left( \ell^2n + k\ell + \tfrac{1-\ell^2}{24} \right)\equiv0\pmod{2},
\end{align*}
for any prime $\ell$ satisfying $\ell\not\equiv3,23\pmod{24}$ and $\ell\nmid k$
\cite[Theorem 1.4]{BrysonOnoPitmanRhoades1}.

Both partitions and strongly unimodal sequences have a statistic defined
on them called the rank. The {\it rank of a partition} is the largest
part minus the number of parts. The {\it rank of a strongly unimodal
sequence} is the number of terms after the peak minus the number of terms
before the peak. 
We note that the peak is unique for a strongly unimodal sequence,
and so there is no ambiguity in this definition as we might 
have with ordinary unimodal sequences. 
We let $N(m,n)$ denote the number of partitions of size $n$ with rank $m$ and
let $u(m,n)$ denote the number of strongly unimodal sequences of size $n$ with 
rank $m$.
Again by standard counting techniques
we have that the relevant generating functions are given by
\begin{align*}
R(\zeta;q) 
&:= 
	\sum_{\substack{n\ge0\\ m\in\mathbb{Z}}} N(m,n)\zeta^m q^n 
	=
	\sum_{n\ge0} \frac{q^{n^2}}{\left(\zeta q,\zeta^{-1}q;q\right)_{n}}
,\\
U(\zeta;q) 
&:= 
	\sum_{\substack{n\ge1\\ m\in\mathbb{Z}}} u(m,n)\zeta^mq^n 
	=
	\sum_{n\ge1} \left(-\zeta q,-\zeta^{-1}q;q\right)_{n-1} q^{n}
.
\end{align*}
The rank for strongly unimodal sequences is somewhat new and first appeared
in \cite{BrysonOnoPitmanRhoades1}. However, the rank of partitions
has a longer history. It was introduced by Dyson \cite{Dyson1} in an attempt
to provide a combinatorial refinement for the Ramanujan congruences modulo $5$ and $7$, which came to full
fruition in \cite{AtkinSwinnertonDyer1}. 
Both $R(\zeta;q)$ and $U(\zeta;q)$ are of considerable interest because of their
modular properties. All of Ramanujan's third order mock theta functions
are specializations of $R(\zeta;q)$ with $\zeta$ taken to be a root of unity
multiplied by a fractional power of $q$; 
that specializations of this type are always mock modular forms
was established in \cite{BringmannOno1}. 
Also, $U(-1;q)$ is one of the most well known examples of a quantum modular form,
which are explained at the end of Section 4.

A key link between the rank of partitions and the rank of strongly unimodal
sequences is the relation between the summands in their generating functions. In
particular,
\begin{align}\label{Eq:RankSummands}
\frac{q^{(-n)^2}}{\left(\zeta q,\zeta^{-1}q;q\right)_{-n}}
&=
	\left(\zeta,\zeta^{-1};q\right)_nq^{n}
.
\end{align}
It is through this connection that one can easily explain the mock modular
properties of $U(\zeta;q)$. Specifically, using a certain $_2\psi_2$ identity
(see equation (3.2.2) and  entry 3.4.7 in \cite{AndrewsBerndt2}), 
we have
\begin{align}\label{Eq:UnimodalToPartitionRank}
\left(1+\zeta\right)\left(1+\zeta^{-1}\right)U(\zeta;q)
&= 
	-R(-\zeta;q) 
	+ 
	\frac{1+\zeta^{-1}}{(q;q)_\infty}
	\sum_{n\in\mathbb{Z}}\frac{\zeta^nq^{\frac{n(n+1)}{2}}}{1+\zeta^{-1}q^n}
.
\end{align}
Due to work of Zwegers \cite{Zwegers1,Zwegers2}, the mock modular properties of the 
functions on the right hand-side of \eqref{Eq:UnimodalToPartitionRank}
are well understood.

To introduce new restricted unimodal sequences, we take the relation
in \eqref{Eq:RankSummands} as the guiding principle. We recall three additional
well-known rank functions are defined by
\begin{gather*}
\overline{R}(\zeta;q)
:=
	\sum_{n\ge0}\frac{(-1;q)_nq^{\frac{n(n+1)}{2}}}{\left(\zeta q,\zeta^{-1}q;q\right)_n}
,\qquad
\overline{R2}(\zeta;q)
:=
	\sum_{n\ge0}\frac{(-1;q)_{2n}q^{n}}{\left(\zeta q^2,\zeta^{-1}q^2;q^2\right)_n}	
,\\
R2(\zeta;q)
:=
	\sum_{n\ge0}\frac{\left(-q;q^2\right)_{n}q^{n^2}}{\left(\zeta q^2,\zeta^{-1}q^2;q^2\right)_n}	
.
\end{gather*}
Respectively, these are the generating functions of the Dyson rank of overpartitions \cite{Lovejoy3}, the 
$M_2$-rank of overpartitions \cite{Lovejoy4}, and the $M_2$-rank of partitions without
repeated odd parts \cite{BerkovichGarvan1,LovejoyOsburn1}. 
For completeness, an {\it overpartition of size $n$} is a partition of size $n$
where the last appearance of each part may (or may not) be overlined.
We replace $n$ with $-n$ in the summands of the generating functions above
and are led to the following three definitions:
\begin{align}
\label{Eq:FOneDef}
\FOne(\zeta;q)
&=
	\sum_{\substack{n\ge0\\m\in\mathbb{Z}}} \fone(m,n)\zeta^mq^n
	:=
	\sum_{n\geq 1} \frac{\left(-\zeta q,-\zeta^{-1}q;q\right)_{n-1}q^n}{(-q;q)_n}
,\\
\notag
\FTwo(\zeta;q)
&=
	\sum_{\substack{n\ge0\\m\in\mathbb{Z}}} \ftwo(m,n)\zeta^m(-1)^nq^n
	:=
	\sum_{n\geq 1} \frac{\left(-\zeta q^2,-\zeta^{-1}q^2;q^2\right)_{n-1}q^{2n}}{(-q;q)_{2n}}
,\\
\notag
\FThree(\zeta;q)
&=
	\sum_{\substack{n\ge0\\m\in\mathbb{Z}}} \fthree(m,n)\zeta^m(-1)^nq^n
	:=
	\sum_{n\geq 1} \frac{\left(-\zeta q^2,-\zeta^{-1}q^2;q^2\right)_{n-1}q^{2n}}{\left(-q;q^2\right)_n}
.
\end{align}
The need for the factor $(-1)^n$ in the definitions of $\FTwo$ and $\FThree$
becomes apparent below when we give the combinatorial interpretations,
which are given at the beginning of Sections 3, 4, and 5.
We also consider the $\zeta=1$ cases of these functions and set
\begin{gather*}
\FOne(q) := \FOne(1;q) =: \sum_{n\ge0} \fone(n)q^n
,\qquad
\FTwo(q) := \FTwo(1;-q)=: \sum_{n\ge0} \ftwo(n)q^n
,\\
\FThree(q) := \FThree(1;-q)=: \sum_{n\ge0} \fthree(n)q^n
.
\end{gather*}

Unimodal sequence type ranks of a similar shape were introduced
by Kim, Lim, and Lovejoy \cite{KimLimLovejoy1}. Their functions are
given by
\begin{gather*}
V(\zeta;q) 
:= 
	\sum_{n\geq 0} \frac{\left(-\zeta q,-\zeta^{-1}q;q\right)_nq^n}{\left(q;q^2\right)_{n+1}}
,\qquad 
W(\zeta;q) 
:= 
	\sum_{n\geq 0} \frac{\left(\zeta q,\zeta^{-1}q;q^2\right)_nq^{2n}}{(-q;q)_{2n+1}}
,\\
Z(\zeta;q) 
:= 
	\sum_{n\geq 0} \frac{\left(-\zeta q,-\zeta^{-1}q;q\right)_nq^{n}}{(q;q)_{2n+1}}
.
\end{gather*}
To recall one of the combinatorial interpretations, let
\begin{align*}
V(\zeta;q) &=: \sum_{\substack{n\ge0\\  m\in\mathbb{Z}}}v(m,n)\zeta^mq^n.
\end{align*}
Then $v(m,n)$ is the number of odd-balanced
unimodal sequences of $2n+2$ with rank $m$. A unimodal sequence
being \textit{odd-balanced} means that the peak is even, the subsequence
of even parts is strongly unimodal, and each odd part appears to
the left of the peak exactly as many times as it appears to the right
of the peak. We note that since the odd parts appear identically on the left and
right of the peak, the rank of an odd-balanced unimodal sequence
is equal to the rank of the subsequence of even parts. 
Kim, Lim, and Lovejoy investigated these functions in terms of their mock modular
and quantum modular behavior, as well as giving some parity results. The
functions $V(\zeta;q)$ and $W(\zeta;q)$
were further studied by Barnett, Folsom, Ukogu, Wesley, and Xu \cite{BarnettFolsomUkoguWesleyXu1}
for their mock and quantum modular properties.
Our first result gives the mock modular properties of
$\FOne(\zeta;q)$, $\FTwo(\zeta;q)$, and $\FThree(\zeta;q)$.

\begin{theorem}\label{TheoremMock}
The functions $\FOne(\zeta;q)$, $\FTwo(\zeta;q)$, and $\FThree(\zeta;q)$,
if $\zeta$ is specialized to a root of unity times a fractional power of $q$, 
are essentially mixed mock modular forms.
\end{theorem}

Specifically, Theorem \ref{TheoremMock} follows from Corollaries 
\ref{CorollaryFOneMockModular}, \ref{cor:F2}, and \ref{CorollaryF3MockModular}.

The next theorem gives the asymptotic behavior of $\ftwo(n)$ and $u2(n)$ as $n \to \infty$.

\begin{theorem}\label{TheoremAsymptotics} We have, as $n\rightarrow\infty$, 
\begin{align*}
\ftwo(n) &\sim \frac{1}{8(2n)^{\frac34}} e^{\pi\sqrt{\frac{n}{2}}}
,\qquad\qquad\qquad
\fthree(n) \sim \frac{1}{4\sqrt{3}(6n)^{\frac34}}e^{\pi\sqrt{\frac{2n}{3}}}
.
\end{align*}
\end{theorem}

Additionally, we fully determine the parity of $\fthree(n)$.
\begin{theorem}\label{thm:f3odd}
We have that $\fthree(n)$ is odd if and only if $8n-1=3^b\ell^2p^c$, where $p$ 
is a prime congruent to $5$ or $23$ modulo $24$, $p\nmid \ell$, 
$b\in\N_0$, $\ell\in\N$, and $c\equiv1\pmod{4}$.
\end{theorem}

We note that the generating function $\FThree(\zeta;q)$ was simultaneously and independently
introduced by Barnett, Folsom, and Wesley \cite{BarnettFolsomWesley1}. There the
relevant function is $N(z;\tau)$, which the authors study for its mock and quantum Jacobi
properties. 
Furthermore, Theorem \ref{thm:f3odd} given above 
was independently discovered and given as Conjecture 1.4 in
\cite{BarnettFolsomWesley1}, and Jeremy Lovejoy has given another proof
in private communications.

The article is organized as follows. In Section 2, we give the various
definitions, identities, and general results required in our proofs. In Section 3,
we discuss the function $\FOne(\zeta;q)$, beginning with its combinatorial interpretation.
As it turns out, this function is the one
for which we can say the least, which is surprising as it comes from the simplest
of the three ranks. The relevant statement of Theorem \ref{TheoremMock} is
contained in Corollary \ref{CorollaryFOneMockModular}.
Section 4 is devoted to the investigation of the function $\FTwo(\zeta;q)$. 
This includes the combinatorial interpretation, identities in terms mock modular objects
for Theorem \ref{TheoremMock} in Corollary \ref{cor:F2}, and 
the asymptotic behavior given in Theorem \ref{TheoremAsymptotics}
is proved toward the end of the section. We also give a brief note on the 
formal dual $\FTwo(\zeta;q^{-1})$
and its quantum modularity. We study $\FThree(\zeta;q)$ in Section 5; again this
includes the combinatorial interpretation, mock modular properties in
Corollary \ref{CorollaryF3MockModular} for Theorem \ref{TheoremMock}, and 
the asymptotics of Theorem \ref{TheoremAsymptotics} are proved after
Corollary \ref{CorollaryF3MockModular}.
In this section we end with a proof of the parity classification 
in Theorem \ref{thm:f3odd} for $\fthree(n)$, which is
related to the arithmetic of $\mathbb{Q}(\sqrt{6})$. We conclude the article with
a few remarks in Section 6.

\section*{Acknowledgments}
We thank Amanda Folsom and Jeremy Lovejoy for bringing \cite{BarnettFolsomWesley1}
to our attention. We also thank the anonymous referees for their helpful comments
and pointing out various typos in an earlier version of this manuscript.

%\section{Preliminary Combinatorial and Asymptotic Results}
%changed here
\section{Preliminaries}
\subsection{Combinatorial results}
We require several known identities and transformation for $q$-series. We state
these results in a series of lemmas. 
In the statements of these identities we give
restrictions for convergence, however we make no mention of this
in our proofs as the convergence conditions are clear and the resulting 
identities hold in greater generality due to analytic continuation.

The first lemma is Heine's transformation.
\begin{lemma}\cite[equation (III.1)]{GasperRahman1}\label{L:Heine}
Suppose that $|t|,|b|,|q|<1$. Then we have
\begin{align*}
\sum_{n\ge0}\frac{(a,b;q)_n}{(c,q;q)_n}t^n
&=
\frac{(b,at;q)_\infty}{(c,t;q)_\infty}
\sum_{n\ge0}\frac{\left(\frac{c}{b},t;q\right)_n}{(at,q;q)_n}b^n.
\end{align*}
\end{lemma}

The following is known as Watson's transformation.
\begin{lemma}\cite[equation (III.18)]{GasperRahman1}\label{L:Watson}
Suppose that $|aq|<|de|$. Then we have
\begin{align*}
\sum_{n\ge0}\frac{\left(\frac{aq}{bc},d,e;q\right)_n\left(\frac{aq}{de}\right)^n}
	{\left(q,\frac{aq}{b},\frac{aq}{c};q\right)_n}
&=
\frac{\left(\frac{aq}{d},\frac{aq}{e};q\right)_\infty}{\left(aq,\frac{aq}{de};q\right)_\infty}
\sum_{n\ge0}
\frac{\left(a,\sqrt{a}q,-\sqrt{a}q,b,c,d,e;q\right)_n  (aq)^{2n} (-1)^nq^{\frac{n(n-1)}{2}}}
{\left(q,\sqrt{a},-\sqrt{a},\frac{aq}{b},\frac{aq}{c},\frac{aq}{d},\frac{aq}{e};q\right)_n (bcde)^n}
.
\end{align*}
\end{lemma}

The next lemma is often used with partial theta functions.
\begin{lemma}\cite[Theorem 6.2.1]{AndrewsBerndt2}\label{L:AndrewsBerndtTheorem621}
Suppose that $|b|<1$ and $|Abq|<|a|$. Then we have
\begin{align*}
&\sum_{n\ge0}\frac{(B,-Abq;q)_nq^n}{(-aq,-bq;q)_n}
	\\\nonumber
&=
	-\frac{(B,-Abq;q)_\infty}{a(-aq,-bq;q)_\infty}
	\sum_{n\ge0}\frac{\left(A^{-1};q\right)_n \left(\frac{Abq}{a}\right)^n}{\left(-\frac{B}{a};q\right)_{n+1}}
	+
	(1+b)
	\sum_{n\ge0}\frac{\left(-a^{-1};q\right)_{n+1} \left(-\frac{ABq}{a};q\right)_n (-b)^n}
	{\left(-\frac{B}{a}, \frac{Abq}{a};q\right)_{n+1}}
.
\end{align*}
\end{lemma}

Furthermore we use another identity related to Lemma \ref{L:AndrewsBerndtTheorem621}.
\begin{lemma}\cite[entry 6.3.12]{AndrewsBerndt2}\label{L:Entry6.3.12}
The following identity holds,
\begin{align*}
\sum_{n\ge0} \frac{(-aq,-bq;q)_nq^{n+1}}{(-cq;q)_n}
&=
	\sum_{n\ge1} \frac{\left(-c^{-1};q\right)_n \left(\frac{ab}{c}\right)^{n-1} q^{\frac{n(n+1)}{2}}}
		{\left(\frac{aq}{c},\frac{bq}{c};q\right)_n}
	-
	\frac{(-aq,-bq;q)_\infty}{c(-cq;q)_\infty}
	\sum_{n\ge1} \frac{ \left(\frac{ab}{c^2}\right)^{n-1} q^{n^2}}
		{\left(\frac{aq}{c},\frac{bq}{c};q\right)_n}
.
\end{align*}
\end{lemma}

We also make use of the Bailey pair machinery \cite[Chapter 3]{Andrews2}.
Recall that a pair of sequences $(\alpha_n,\beta_n)$ is called a {\it Bailey pair} relative to 
$(a,q)$ if
\begin{align*}
\beta_n&=\sum_{0\le j\le n}\frac{\alpha_j}{(q;q)_{n+j}(aq;q)_{n-j}}.
\end{align*}
Bailey's Lemma is as follows.
\begin{lemma}\cite[Theorem 3.4]{Andrews2}\label{L:BaileysLemma}
If $(\alpha_n,\beta_n)$ is a Bailey pair relative to 
$(a,q)$, then, assuming convergence conditions, 
\begin{align*}
\sum_{n\ge0}(\varrho_1,\varrho_2;q)_n\left(\frac{aq}{\varrho_1\varrho_2}\right)\beta_n
&=
\frac{ \left( \frac{aq}{\varrho_1},\frac{aq}{\varrho_2};q\right)_\infty}
{\left( aq, \frac{aq}{\varrho_1\varrho_2} ;q\right)_\infty}
\sum_{n\ge0}\frac{ (\varrho_1,\varrho_2;q)_n\left(\frac{aq}{\varrho_1\varrho_2}\right)}
{\left( \frac{aq}{\varrho_1},\frac{aq}{\varrho_2};q\right)_n }\alpha_n.
\end{align*}
\end{lemma}

The following theorem of Lovejoy gives a convenient formula for constructing Bailey pairs.
\begin{lemma} \cite[Theorem 8]{Lovejoy1}\label{LemmaLovejoy}
The following  is a Bailey pair relative to  $(a,q)$:
\begin{align*}
\alpha_n 
&=
	\frac{\left( \frac{a}{b},\frac{a}{c},\frac{a}{d};q \right)_n \left(1-aq^{2n}\right) (-bcdq)^nq^{\frac12n(n-1)} }
	{(1-a)(bq,cq,dq;q)_n a^n}
	\sum_{0\leq j \leq n}
	\frac{(a;q)_{j-1} (b,c,d;q)_j \left(1-aq^{2j-1}\right)a^j }
		{\left( q,\frac{a}{b},\frac{a}{c},\frac{a}{d};q\right)_j (bcd)^j }\\
\beta_n &= \frac{\left( \frac{adq}{bc};q \right)_n}{(bq,cq,dq;q)_n}		
.
\end{align*}
\end{lemma}

\subsection{Analytic results}\label{SectionPrelimAnalysis}
To recognize the functions of interest for this paper as mixed mock modular forms, we recall
a few well-known functions.
For $z\in\mathbb{C}$, define the Jacobi theta function
\begin{equation*}
\vartheta(z;\tau) 
:= 
	\sum_{n\in\frac12+\Z} e^{2\pi in\left(z+\frac12\right)} q^{\frac{n^2}{2}}
	 = -iq^{\frac18} \zeta^{-\frac12} \left(q,\zeta,\zeta^{-1}q;q\right)_\infty
,
\qquad(\zeta:=e^{2\pi iz}).
\end{equation*}
Moreover we require Zwegers $\mu$-function for $z_1,z_2\in\C$,
\begin{equation*}
\mu(z_1,z_2;\tau) := \frac{e^{\pi iz_1}}{\vartheta(z_2;\tau)} \sum_{n\in\Z} \frac{(-1)^n e^{2\pi inz_2}q^{\frac{n(n+1)}{2}}}{1-e^{2\pi iz_1}q^n},
\end{equation*}
and for $\ell\in\N$ let the higher level Appel function be given by
\begin{equation*}
A_\ell(z_1,z_2;\tau) := e^{\pi i\ell z_1} \sum_{n\in\Z} \frac{(-1)^{\ell n} e^{2\pi inz_2} q^{\frac{\ell n(n+1)}{2}}}{1-e^{2\pi iz_1}q^n}.
\end{equation*}
The function $\vartheta(z;\tau)$ is a holomorphic Jacobi form
and the mock modular properties of
$\mu(z_1,z_2;\tau)$ and $A_\ell(z_1,z_2;\tau)$  are described in \cite{Zwegers1,Zwegers2}.

To prove Theorem \ref{TheoremAsymptotics},
we also require the following asymptotic behavior which follows directly from the 
modular transformation of the Dedekind $\eta$-function
\begin{equation}\label{Pochas}
\left(e^{-w};e^{-w}\right)_\infty \sim \sqrt{\tfrac{2\pi}{w}} e^{-\frac{\pi^2}{6w}} \qquad \mbox{as }w\to 0,
\end{equation}
where the limit is taken in any region $|\!\Arg(w)|<\theta$, for fixed $\theta<\frac{\pi}{2}$.
Moreover we need the following Tauberian Theorem. 
\begin{theorem}\cite[Theorem 14.4]{BringmannFolsomOnoRolen1}\label{Taub}
Let $f(q)=\sum_{n\geq 0} a(n) q^n$ be a power series with non-negative $a(n)$ 
that are monotonically increasing and have radius of convergence equal to $1$. 
Suppose that	
\begin{align*}
f\left(e^{-t}\right) &\sim \lambda t^\alpha e^{\frac{A}{t}} 
	&&\mbox{as }t\to 0^+
,\\
f\left(e^{-w}\right) &\ll  |w|^\alpha e^{\frac{A}{|w|}} 
	&&\mbox{as }w\to 0
	\mbox{ in each region } \left|\Arg(w)\right| < \theta < \tfrac{\pi}{2}
,
\end{align*}
for $A>0$, $\lambda,\alpha\in\R$. Then we have
\begin{equation*}
a(n) \sim \frac{\lambda}{2\sqrt{\pi}} \frac{A^{\frac{\alpha}{2}+\frac14}}{n^{\frac{\alpha}{2}+\frac34}} e^{2\sqrt{An}}
\qquad \mbox{as }n\to\infty.
\end{equation*}
\end{theorem}
\begin{remark}
Theorem \ref{Taub} is commonly stated without the additional boundedness condition.
However, this seems to be in error and is discussed in detail in an upcoming
article \cite{BringmannJenningsShafferMahlburg1}. In the current article, we determine
the asymptotic behavior of functions via modular transformations, which actually imply
$f(e^{-w}) \sim \lambda w^\alpha e^{\frac{A}{w}}$ 
as $w\rightarrow0$ in each region $|\!\Arg(w)|<\theta<\frac{\pi}{2}$
(and as such, the required bound holds), so that this detail is not of major concern.
\end{remark}

\section{The function $\FOne(\zeta;q)$}
Since the series expansions of the summands of \eqref{Eq:FOneDef}
have both positive and negative coefficients, we interpret $\fone(n)$
and $\fone(m,n)$ both as the difference of two non-negative counts.
A {\it left-heavy overlined} unimodal sequence of size $n$ is a
unimodal sequence of size $n$ such that the parts up to and including 
all occurences of the peak form an overpartition with largest part overlined,
and the parts after the peak form an overpartition with all parts overlined.
Then $\fone(n)$ is the number of
left-heavy overlined unimodal sequences of size $n$ with
an even number of non-overlined parts minus those with an odd
number of non-overlined parts. Furthermore, $\fone(m,n)$ is the same
difference of counts as $\fone(n)$, but with the added restraint
that the rank of the strongly unimodal sequence consisting of the overlined parts is $m$.

\begin{example}
The left-heavy overlined unimodal sequences of $3$ are
$(\overline{3})$, $(1,\overline{2})$, 
$(\overline{1},\overline{2})$,
$(\overline{2},\overline{1})$, and
$(1,1,\overline{1})$. By accounting for the parity of the non-overlined parts,
we find $\fone(3)=3$.
The ranks of the strongly unimodal subsequences 
consisting of the overlined parts
are, respectively, $0$, $0$, $-1$, $1$, and $0$. 
\end{example}

The following lemma writes $\FOne(\zeta;q)$ in terms of
$R(\zeta;q)$ and $\overline{R}(\zeta;q)$.
\begin{lemma}\label{LemmaFOneMock1}
We have
\begin{align*}
(1-\zeta)\left(1-\zeta^{-1}\right)\FOne(\zeta;q)
&=
\overline{R}(\zeta;q)
-
\frac{\left(-\zeta q,-\zeta^{-1}q;q\right)_\infty}{(-q;q)_\infty}R(\zeta;q)
.
\end{align*}
\end{lemma}
\begin{proof}
Lemma \ref{L:Entry6.3.12}
gives, shifting $n \mapsto n+1$ in the definition of $\FOne(\zeta;q)$, 
\begin{align*}
(1+q)\FOne(\zeta;q)
&=
	\sum_{n\ge0}\frac{\left(-\zeta q,-\zeta^{-1}q;q\right)_nq^{n+1}}{\left(-q^2;q\right)_n}
\\
&=
	\sum_{n\ge1} \frac{(1+q)(-1;q)_{n-1} q^{\frac{n(n-1)}{2}}}{\left(\zeta,\zeta^{-1};q\right)_n}
	-
	\frac{\left(-\zeta q,-\zeta^{-1}q;q\right)_\infty}{\left(-q^2;q\right)_\infty}
	\sum_{n\ge1} \frac{ q^{(n-1)^2}}{\left(\zeta,\zeta^{-1};q\right)_n}
\\
&=
	\frac{1+q}{(1-\zeta)\left(1-\zeta^{-1}\right)}\overline{R}(\zeta;q)
	-
	\frac{(1+q)\left(-\zeta q,-\zeta^{-1}q;q\right)_\infty}{(1-\zeta)\left(1-\zeta^{-1}\right)(-q;q)_\infty}
	R(\zeta;q),
\end{align*}
shifting $n \mapsto n+1$ in the definitions of 
$\overline{R}(\zeta;q)$ and $R(\zeta;q)$.
This gives the claim.
\end{proof}

It is not hard to conclude the following representation using (mock) modular objects,
which are defined in Section \ref{SectionPrelimAnalysis}.

\begin{corollary}\label{CorollaryFOneMockModular}
We have
\begin{align*}
\FOne(\zeta;q)
&=
	-\frac{2\zeta \eta(2\tau) A_2\left(z,\frac12;\tau\right)}{\left(1-\zeta^2\right)\eta(\tau)^2}
	-
	\frac{\vartheta\left(z+\frac12;\tau\right)  A_3(z,-\tau;\tau)}
	{\left(1-\zeta^2\right)\eta(\tau)\eta(2\tau)}
	-
	\frac{\zeta}{1-\zeta^2}
.
\end{align*}
\end{corollary}

\section{The function $\FTwo(\zeta;q)$}
Before we state the combinatorial interpretation of $\ftwo(n)$
and $\ftwo(m,n)$, note that the summands of $\FTwo(\zeta;-q)$ have
non-negative coefficients since
\begin{align*}
\frac{1}{(q;-q)_{2n}}
&=
	\frac{(-q;q^2)_n}{(q^{2n+2};q^2)_n}
.
\end{align*}
A unimodal sequence is an {\it $M_2$-left-heavy overlined} 
unimodal sequence
if the
peak is even and appears overlined exactly once (suppose it is $\overline{2N}$),
the parts before and after $\overline{2N}$ form an overpartition, 
all overlined odd parts appear to the left of $\overline{2N}$,
and all non-overlined parts are at least $N+1$ and
 appear identically on the left and right of $\overline{2N}$.
Then $\ftwo(n)$ is the number of $M_2$-left-heavy overlined
unimodal sequences of size $n$ and 
$\ftwo(m,n)$ is the number of $M_2$-left-heavy overlined
unimodal sequences of size $n$ such that the strongly
unimodal sequence consisting of the overlined even parts has rank $m$.

\begin{example}
We have $\ftwo(7)=5$ since the relevant sequences are:
$(\overline{1},\overline{6})$,
$(\overline{3},\overline{4})$,
$(\overline{1},\overline{2},\overline{4})$,
$(\overline{1},\overline{4},\overline{2})$, and
$(\overline{1},2,\overline{2},2)$.
The residual ranks of these sequences are 
$0$, $0$, $-1$, $1$, and $0$, respectively.
\end{example}

The following proposition rewrites $\FTwo(\zeta;-q)$ in terms of generalized Lambert series.
\begin{proposition}\label{prop:F2}
We have
\begin{align*}
\FTwo(\zeta;-q)&=
	-\frac{\zeta q\left(-\zeta q^2,-\zeta^{-1}q^2,-q;q^2\right)_\infty}{(1-\zeta)(q;q)_\infty\left(-q^2;q^2\right)_\infty}
	\sum_{n\in\Z} \frac{(-1)^n q^{2n^2+3n}}{1+\zeta q^{2n+1}}
	\\
	&\quad
	+
	\frac{\zeta^{2}\left(q^2;q^4\right)_\infty}{2\left(1-\zeta^2\right)\left(q^4;q^4\right)_\infty}
	\sum_{n\in\mathbb{Z}}
	\frac{q^{n^2+3n+1}}{1+\zeta q^{2n+1}}
	-
	\frac{\zeta\left(q^2;q^4\right)_\infty}{2\left(1-\zeta^2\right)\left(q^4;q^4\right)_\infty}
	\sum_{n\in\mathbb{Z}}
	\frac{q^{n^2+n}}{1+\zeta q^{2n+1}}	
	.	
\end{align*}
\end{proposition}
\begin{proof}
Applying Lemma \ref{L:AndrewsBerndtTheorem621}
with $q\mapsto q^2, a=-q, b=q^2, A=\zeta^{-1}q^{-2}$, and $B=-\zeta q^2$, we find that
\begin{align}\label{E:F2Sums1}
&q^{-2}(1-q)\left(1+q^2\right)\FTwo(\zeta;-q)
=
	\sum_{n\ge0}\frac{\left(-\zeta q^2,-\zeta^{-1}q^2;q^2\right)_nq^{2n}}{\left(q^3,-q^4;q^2\right)_n} \notag
\\
&=
	\frac{q^{-1}\left(-\zeta q^2,-\zeta^{-1}q^2;q^2\right)_\infty}{(1+\zeta q)\left(q^3,-q^4;q^2\right)_\infty}
	\sum_{n\ge0}\frac{\left(\zeta q^2;q^2\right)_n(-1)^n q^n \zeta^{-n}}{\left(-\zeta q^3;q^2\right)_{n}} \notag
	\\&\hspace{6cm}
	-
	\frac{q^{-1}(1-q)\left(1+q^2\right)}{(1+\zeta q)\left(1+\zeta^{-1}q\right)}
	\sum_{n\ge0}
	\frac{\left(q^2;q^4\right)_{n} (-1)^n q^{2n}}
	{\left(-\zeta q^3,-\zeta^{-1}q^3;q^2\right)_{n}}
.
\end{align}

We handle the two sums in \eqref{E:F2Sums1} separately. 
For the first, we apply Lemma \ref{L:Heine} with 
$q\mapsto q^2$, $a=\zeta q^2$, $b=q^2$, $c=-\zeta q^3$, and $t=-\zeta^{-1}q$, which gives that
\begin{align}\label{E:F2E1}
\sum_{n\ge0}\frac{\left(\zeta q^2;q^2\right)_n(-1)^n q^n \zeta^{-n}}{\left(-\zeta q^3;q^2\right)_{n}}
&=
	\frac{\left(q^2,-q^3;q^2\right)_\infty}{\left(-\zeta q^{3},-\zeta^{-1}q;q^2\right)_\infty}
	\sum_{n\ge0} \frac{\left(-\zeta q,-\zeta^{-1}q;q^2\right)_n q^{2n}}{\left(-q^3,q^2;q^2\right)_n}
.
\end{align}
Next we take $q\mapsto q^2$, $a=q^2$, $b=-q$, $c\rightarrow\infty$, $d=-\zeta q$, and $e=-\zeta^{-1}q$
in Lemma \ref{L:Watson} to find that
\begin{align}\label{E:F2E2}
\sum_{n\ge0}\frac{\left(-\zeta q,-\zeta^{-1}q;q^2\right)_nq^{2n}}{\left(q^2,-q^3;q^2\right)_n}=
	\frac{-\zeta\left(-\zeta q,-\zeta^{-1}q;q^2\right)_\infty}{(1-q)(1-\zeta)\left(q^4,q^2;q^2\right)_\infty}
	\sum_{n\in\Z} \frac{(-1)^n q^{2n^2+3n}}{1+\zeta q^{2n+1}}.
\end{align}
Combining \eqref{E:F2E1} and \eqref{E:F2E2} gives the first summand in the proposition.

For the the second series in \eqref{E:F2Sums1} we apply Lemma \ref{L:Watson}
with $q\mapsto q^2$, $a=q^2$, $b=-\zeta q$, $c=-\zeta^{-1}q$, $d=q$, and $e=-q$. 
This gives that
\begin{align*}
&\sum_{n\ge0}\frac{\left(q^2;q^4\right)_n(-1)^nq^{2n}}{\left(-\zeta q^3,-\zeta^{-1}q^3;q^2\right)_n}
=
	\frac{(1+\zeta q)\left(1+\zeta^{-1}q\right)\left(q^6;q^4\right)_\infty}{2\left(-q^2,q^4;q^2\right)_\infty}
	\sum_{n\in\mathbb{Z}}
	\frac{q^{n^2+3n}}{\left(1+\zeta q^{2n+1}\right) \left(1+\zeta^{-1}q^{2n+1}\right)}
\\
&=
	\frac{(1+\zeta q)\left(1+\zeta^{-1}q\right)\left(q^6;q^4\right)_\infty}{2\left(1-\zeta^2\right)\left(-q^2,q^4;q^2\right)_\infty}
	\sum_{n\in\mathbb{Z}}
	\frac{q^{n^2+3n}}{1+\zeta^{-1}q^{2n+1}}
	\\&\hspace{7cm}
	-
	\frac{\zeta^{2}(1+\zeta q)\left(1+\zeta^{-1}q\right)\left(q^6;q^4\right)_\infty}{2\left(1-\zeta^2\right)\left(-q^2,q^4;q^2\right)_\infty}
	\sum_{n\in\mathbb{Z}}
	\frac{q^{n^2+3n}}{1+\zeta q^{2n+1}}
.
\end{align*}
Letting $n\mapsto -n-1$ in the last sum and combining terms gives the claim.
\end{proof}

We next rewrite $\FTwo$ in terms of (mock) modular objects. 
\begin{corollary}\label{cor:F2}
We have
\begin{multline*}
\FTwo(\zeta;-q) 
= 
	-\frac{\zeta^{\frac12}\eta(2\tau)^2 \vartheta\left(z+\frac12;2\tau\right)}{\left(1-\zeta^2\right)\eta(\tau)^2\eta(4\tau)^2}
		A_2\left(z+\frac12+\tau,\frac12+\tau;2\tau\right)
	\\
	- \frac{i\zeta^{\frac12}q^{-\frac14}}{1-\zeta^2}\left(2\mu\left(z+\frac12+\tau,\frac12;2\tau\right)+i\zeta^{\frac12}q^{\frac14}\right).
\end{multline*}
\end{corollary}

The asymptotic formula for $\ftwo(n)$ in Theorem \ref{TheoremAsymptotics} follows
from Theorem \ref{Taub}, once we establish that
$\ftwo(n)$ is monotonic.
\begin{lemma} 
For $n\in \N_0$, we have $\ftwo(n+1)\ge \ftwo(n)$.
\end{lemma}
\begin{proof}
To prove the claim, we show that
$(1-q)\FTwo(1;-q)$ has non-negative coefficients. For this, we note that
\begin{align*}
(1-q)\FTwo(1;-q)
&=
\sum_{n\ge1}F_n(q),\qquad \textnormal{where }F_{n}(q) := \frac{\left(-q^2;q^2\right)_{n-1}q^{2n}}{\left(1+q^{2n}\right) \left(q^3;q^2\right)_{n-1}}.
\end{align*}
We first verify that $F_{n}(q)$ has 
non-negative coefficients for $n\ge3$. 
Given two power series $A(q)$ and $B(q)$, write
$A(q)\succeq B(q)$ to indicate that $A(q)-B(q)$ has non-negative coefficients.
We see that
\begin{gather*}
G_n(q) = \sum_{m\ge0} (g_{n,o}(m)-g_{n,e}(m))q^m,
\qquad\mbox{where }\\[-0.5ex]
G_{3}(q) := \frac{q^{6}}{\left(1-q^3\right)\left(1+q^6\right)},\qquad \qquad 
G_n(q) := \frac{q^{2n}}{\left(1-q^3\right)\left(1-q^{2n-3}\right)\left(1+q^{2n}\right)}
\,\,\mbox{ for } n\ge4,
\end{gather*}
and $g_{n,o}(m)$ ($g_{n,e}(m)$, resp.\@) is the number of 
partitions of $m$ with largest part $2n$, where the only other allowed parts
are $3$ and $2n-3$, and $2n$ appears an odd (even, resp.\@) number of 
times. 
Taking an occurrence of $2n$ and replacing it by
$3$ and $2n-3$ gives an injection from the partitions counted by $g_{n,e}(m)$
to those of $g_{n,o}(m)$. Thus $G_n(q)\succeq0$, which implies 
$F_{n}(q)\succeq0$ for $n\ge3$. However, $F_{1}(q)$ and $F_{2}(q)$ do have negative coefficients.
By a careful grouping of $F_{1}(q)+F_{2}(q)+F_{3}(q)+F_{4}(q)$ as rational
functions, we find  $F_{1}(q)+F_{2}(q)+F_{3}(q)+F_{4}(q)\succeq0$.
Thus $\ftwo(n+1)\ge \ftwo(n)$ for all $n$. We carefully group $F_{1}(q)+F_{3}(q)$ and $F_{2}(q)+F_{4}(q)$ as follows,
\begin{align*}
F_{1}(q)+F_{3}(q)
&=
	\frac{\left(1+q^4+q^9\right)q^2}{1-q^{12}}
	+
	\frac{\left(q^4+q^7+2q^8+2q^{11}+q^{15}+q^{19} \right)q^2}{\left(1-q^5\right)\left(1-q^{12}\right)}
	- q^4
,\\
F_{2}(q)+F_{4}(q)
&=
	\frac{\left(1+q^2\right)q^{4}}{\left(1-q^3\right)\left(1+q^4\right)}
	+
	\frac{\left(1+q^2\right)\left(1+q^4\right)\left(1+q^6\right)q^{8}}{\left(1-q^3\right)\left(1-q^5\right)\left(1-q^7\right)\left(1+q^8\right)}
\\
&\succeq
	\frac{\left(1+q^2\right)q^{4}}{\left(1-q^3\right)\left(1+q^4\right)}
	+
	\frac{\left(1+q^2\right)\left(1+q^4\right)\left(1+q^6\right)q^{8}}{\left(1-q^3\right)\left(1-q^5\right)\left(1+q^8\right)}
\\
&=
	\frac{\left(1+2q+q^2+q^4+q^7+q^8+q^{10}\right)q^{13}}{\left(1-q^5\right)\left(1-q^{16}\right)}
	+
	\frac{\left(1+q^2+2q^8+q^{10}+2q^{13}+q^{19}\right)q^4}
	{\left(1-q^3\right)\left(1-q^{16}\right)}
,
\end{align*}
where we make use of the fact that $G_4(q)\succeq0$.
We then find $F_{1}(q)+F_{2}(q)+F_{3}(q)+F_{4}(q)\succeq0$,
and thus $\ftwo(n+1)\ge \ftwo(n)$ for all $n$.
\end{proof}

The following calculation gives the asymptotics of $\ftwo(n)$.
\begin{proof}[Proof of the asymptotics for $\ftwo(n)$ in Theorem \ref{TheoremAsymptotics}]
We begin with the representation in \eqref{E:F2Sums1},
\begin{gather*}
\FTwo(1;-q)
=
	\frac{q\left(-q^2;q^2\right)_\infty}{\left(q;q^2\right)_\infty}\sum_{n\ge0}\frac{\left(q^2;q^2\right)_n(-1)^nq^n}{\left(-q;q^2\right)_{n+1}}
	-
	q\sum_{n\ge0}\frac{\left(q^2;q^4\right)_n(-1)^nq^{2n}}{\left(-q;q^2\right)_{n+1}^2}
.
\end{gather*}
If $q=e^{-w}$ and $w\rightarrow0$ in a region where $|\!\Arg(w)|<\theta$,
we have $q\rightarrow1$ and the sums become $\frac12$ and $\frac14$ respectively.
Thus
\begin{align*}
\FTwo(1;-e^{-w}) 
&\sim 
	\frac{\left(-e^{-2w};e^{-2w}\right)_\infty}{2\left(e^{-w};e^{-2w}\right)_\infty}  - \frac14
\sim
	\frac{\left(e^{-4w};e^{-4w}\right)_\infty}{2\left(e^{-w};e^{-w}\right)_\infty}
\qquad
(\mbox{as } w\to 0)
.
\end{align*}
Using \eqref{Pochas} gives that $\FTwo(1;-e^{-w})\sim \frac14 e^{\frac{\pi^2}{8w}}$ as $w\to 0$.
We now use Theorem \ref{Taub} with $\lambda=\frac14$, $\alpha=0$, and $A=\frac{\pi^2}{8}$ to obtain the claim.

For the reader concerned about taking the limit $w\rightarrow0$ inside the sums, one can
instead apply (mock) modular transformations to the representation in Corollary \ref{cor:F2}
to obtain the same results. However, these calculations are considerably longer.
\end{proof}

While $\FTwo(\zeta;-q)$ does not appear to posses any quantum modular properties, the
formal dual $\FTwo(\zeta;-q^{-1})$ does. 
Recall that a function $f:\mathcal{Q}\rightarrow\mathbb{C}$, 
($\mathcal{Q}\subset\mathbb{Q}$) is a  
{\it quantum modular form}, of weight $k$ with respect to $\Gamma$, if
{\it the obstruction to modularity},
$$f(\tau) - \chi(M)^{-1}(c \tau + d)^{-k} f \left(\frac{a\tau + b}{c \tau + d}\right), \quad M=\left(\begin{matrix}
a&b\\ c&d
\end{matrix}\right) \in \Gamma,$$
can be extended to an analytic function on an open subset of $\mathbb{R}$.
Quantum modular forms were introduced by Zagier in \cite{Zagier1}. 

Noting that $( w;q^{-1})_{n} = (w^{-1};q)_n(-1)^nw^nq^{-\frac{n(n-1)}{2}}$,
we obtain
\begin{align*}
\FTwo\left(\zeta ;-q^{-1}\right)
&=
	\sum_{n\ge1}\frac{\left(-\zeta q^2,-\zeta^{-1}q^2;q^2\right)_{n-1}(-1)^nq^n}{(q,-q^2;q^2)_{n}}
=
	\frac{\zeta}{1-\zeta^2}
	\sum_{n\ge1}(-1)^nq^{n^2}(\zeta^{-n}-\zeta^n)
,
\end{align*}
where the second equality is Theorem 15 of \cite{AndrewsWarnaar1} with $q\mapsto-q$.
In particular,
\begin{align}\label{false}
\FTwo\left(1;-q^{-1}\right)&= \sum_{n\geq 1} n(-1)^n q^{n^2}.
\end{align}
Using the now standard techniques for false theta functions \cite{BringmannCreutzigRolen1},
one can show that \eqref{false} is a quantum modular form, where the values of the function
on the rationals are given by taking radial limits.
However, since neither the series for $\FTwo(\zeta;-q)$ nor $\FTwo(\zeta;-q^{-1})$ 
truncates if $q$ is a root of unity, the behavior of one function as $q$ approaches a root
of unity says nothing about the other.

\section{The function $\FThree(\zeta ;q)$}

A unimodal sequence is {\it $M_2$-left heavy} 
if the largest part is even, 
all odd parts appear to the left of the peak,
and the subsequence consisting of the even parts is strongly unimodal.
Then $\fthree(n)$ is the number of $M_2$-left heavy unimodal sequences
of size $n$ and $\fthree(m,n)$ is the number of such sequences
where the strongly unimodal sequence consisting of the
even parts has rank $m$.

\begin{example}
We have $\fthree(6)=5$ since the relevant sequences are:
$(6)$, $(2,4)$, $(4,2)$, $(1,1,4)$, and $(1,1,1,1,2)$.
The residual ranks of these sequences are 
$0$, $-1$, $1$, $0$, and $0$, respectively.
\end{example}

The following proposition rewrites $\FThree(\zeta;-q)$ in terms
of generalized Lambert series. We note that a similar expression for $\FThree(\zeta;-q)$ can be found in
\cite[equation (4.29)]{Mortenson1}.
\begin{proposition}\label{prop:R2}
We have
\begin{align*}
&(1+\zeta)\left(1+\zeta^{-1}\right)\FThree(\zeta;-q)
\\&=
	-R2(-\zeta;-q)
	-\frac{\zeta^{-1}\left(-\zeta,-\zeta^{-1};q^2\right)_\infty}{(1+\zeta q)\left(q;q^2\right)_\infty}R\left(-\zeta q;q^2\right)
	+\frac{(q;q)_\infty \left(q;q^2\right)_\infty^2}{\left(-\zeta q,-\zeta^{-1}q;q\right)_\infty}
	+\frac{\zeta^{-1}\left(-\zeta,-\zeta^{-1};q^2\right)_\infty}{\left(q;q^2\right)_\infty}
	.
\end{align*}
\end{proposition}
\begin{proof}
We use Lemma \ref{L:Entry6.3.12} with $q\mapsto q^2$, $a=\zeta$, $b=\zeta^{-1}$, and $c=-q$, which gives that
\begin{align*}
\FThree(\zeta;-q)
&=
\sum_{n\ge1}\frac{\left(q;q^2\right)_{n-1}(-1)^nq^{n^2}}{\left(-\zeta q,-\zeta^{-1}q;q^2\right)_n}
+
\frac{\left(-\zeta q^2,-\zeta^{-1}q^2;q^2\right)_\infty}{\left(q;q^2\right)_\infty}
\sum_{n\ge1}\frac{q^{2n(n-1)+1}}{\left(-\zeta q,-\zeta^{-1}q;q^2\right)_n}
.
\end{align*}
By entry 12.3.2 of \cite{AndrewsBerndt1}, we obtain that 
\begin{align*}
\sum_{n\ge1}\frac{(-1)^n\left(q;q^2\right)_{n-1}q^{n^2}}{\left(-\zeta q,-\zeta^{-1}q;q^2\right)_n}
&=
-\frac{R2(-\zeta;-q)}{(1+\zeta)\left(1+\zeta^{-1}\right)}
+
\frac{\left(q;q^2\right)_\infty (q;q)_\infty }{\left(-\zeta,-\zeta^{-1},-q;q\right)_\infty}
.
\end{align*}
Using equation (12.2.5) of \cite{AndrewsBerndt1} and  Lemma 7.9 of \cite{Garvan1} yields 
\begin{align*}
\sum_{n\ge1}\frac{q^{2n(n-1)+1}}{\left(-\zeta q,-\zeta^{-1}q;q^2\right)_n}  =
\zeta^{-1}
-\frac{\zeta^{-1}}{1+\zeta q}R\left(-\zeta q;q^2\right),
\end{align*}
giving the claim.
\end{proof}

In the following corollary, 
we rewrite $\FThree(\zeta;-q)$ in terms
of known (mock) modular objects.
\begin{corollary}\label{CorollaryF3MockModular}
We have
\begin{align*}
\FThree(\zeta;-q)
&=
	\frac{q^{\frac18} \eta(\tau) A_2\left(z+\frac12,-\tau;2\tau \right)}{(1+\zeta) \eta(2\tau)^2}
	+
	\frac{i\zeta^{-2}q^{-\frac{13}{8}} \vartheta\left(z+\frac12;2\tau\right) 
		A_3\left(z+\tau+\frac12, -2\tau; 2\tau \right)}
	{(1+\zeta)\eta(\tau)\eta(2\tau)}
	\\&\quad
	-
	\frac{\zeta^{\frac12}q^{\frac{1}{8}} \eta(\tau)^4 }
	{( 1+\zeta) \eta(2\tau)^2 \vartheta\left(z+\frac12;\tau\right)}
	-
	\frac{\zeta^{-\frac12}q^{-\frac{5}{24}} \vartheta\left(z+\frac12;2\tau\right)}
	{(1+\zeta)\eta(\tau)}
.
\end{align*}
\end{corollary}

Next we prove the asymptotic formula for $\fthree(n)$. We note that
$\fthree(n+1)\ge\fthree(n)$ follows trivially by taking a sequence
counted by $\fthree(n)$ and adding a single $1$ to the left
of the peak. Furthermore, this also shows that 
$\fthree(m,n+1)\ge\fthree(m,n)$ for fixed $m$.
\begin{proof}[Proof of the asymptotics for $\fthree(n)$ in Theorem \ref{TheoremAsymptotics}]
By using Proposition \ref{prop:R2}, we have
\begin{align*}
\FThree(1;-q) 
= 
	-\frac14 R2(-1;-q) - \frac{\left(-1;q^2\right)_\infty^2}{4(1+q)\left(q;q^2\right)_\infty} R\left(-q;q^2\right) 
	+ \frac{(q;q)_\infty \left(q;q^2\right)_\infty^2}{4(-q;q)_\infty^2} 
	+ \frac{\left(-1;q^2\right)_\infty^2}{4\left(q;q^2\right)_\infty}.
\end{align*}
With $q=e^{-w}$, $w\rightarrow0$, and $|\!\Arg(w)|<\theta$, 
we find that
\begin{align*}
\frac{\left(-1;q^2\right)_\infty^2}{\left(q;q^2\right)_\infty}
&= 
	4 \frac{\left(q^4;q^4\right)^2_\infty}{(q;q)_\infty \left(q^2;q^2\right)_\infty}
	\sim 
	\sqrt{2} e^{\frac{\pi^2}{6w}}
,\qquad
\frac{(q;q)_\infty \left(q;q^2\right)_\infty^2}{(-q;q)_\infty^2} 	
= 
	\frac{(q;q)_\infty^5}{\left(q^2;q^2\right)_\infty^4} 
	\sim 
	4\sqrt{\frac{2\pi}{w}} e^{-\frac{\pi^2}{2w}}.
\end{align*}
Moreover,
\begin{align*}
\lim_{w\rightarrow0}
R2(-1;-e^{-w})
&=
	\lim_{q\rightarrow1}
 	\sum_{n\geq 0 } \frac{\left(q;q^2\right)_n (-1)^n q^{n^2}}{\left(-q^2;q^2\right)_n^2} 
 	= 1
,\\
\lim_{w\rightarrow0}
R\left(-q;q^2\right) 
&= 
	\lim_{q\rightarrow1}
	\sum_{n\geq 0} \frac{q^{2n^2}}{\left(-q^3,-q;q^2\right)_n} 
	= 
	\sum_{n\geq 0} \left(\frac14\right)^n 
	= \frac43.
\end{align*}	
Thus
\begin{align*}
\FThree(1;-q) 
\sim 
-\frac14+\frac{\left(-1;q^2\right)_\infty^2}{4\left(q;q^2\right)_\infty} \left(-\frac12 R\left(-q;q^2\right) + 1\right) 
\sim  
\frac{1}{6\sqrt{2}} e^{\frac{\pi^2}{6w}}.
\end{align*}	
We now use Theorem \ref{Taub} with $\lambda=\frac{1}{6\sqrt{2}}$, $\alpha=0$, and 
$A=\frac{\pi^2}{6}$	to obtain the claim. 

Again for the reader concerned with
taking limits inside sums, one may instead use modular transformations with the 
representation in Corollary \ref{CorollaryF3MockModular}. While one can save some 
effort by noting $R(-q;q^2)= 1+q-q(1+q)\omega(-q)$, where $\omega(q)$ is a third
order mock theta function, these calculations are still somewhat lengthy.
\end{proof}

A first step to prove Theorem \ref{thm:f3odd} is to rewrite $\FThree(1;q)$ modulo 2.

\begin{proposition}\label{PropositionF3Mod2}
We have that 
\begin{align*}
\FThree(1;-q)
&\equiv
	\sum_{\substack{n\ge0\\0\le j\le n}}
	\left(1+q^{2j+1}\right) q^{3n^2+6n-2j^2-3j+2} 
\pmod{2}.
\end{align*}
\end{proposition}
\begin{proof}
By taking $q\mapsto q^2$, $a=q^4$, $b=q$, $d=c^2$, and then letting
$c\rightarrow 0$ 
in Lemma \ref{LemmaLovejoy}
we have the following Bailey
pair relative to $(q^4,q^2)$,
\begin{align*}
\alpha_n
=
	\frac{(-1)^n \left(1-q^{4n+4}\right) q^{3n^2+4n}(1-q) }{\left(1-q^2\right)\left(1-q^4\right)}
	\sum_{0\leq j \leq n} \left(1+q^{2j+1}\right)q^{-2j^2-3j} 
,\qquad \beta_n = \frac{1}{\left(q^3;q^2\right)_n}.
\end{align*}
We then apply Lemma \ref{L:BaileysLemma}, with $\varrho_1=\varrho_2=q^2$, to this
Bailey pair to obtain 
\begin{align}\label{Thetid}
\sum_{n\ge0}\frac{\left(q^2;q^2\right)_n^2 q^{2n} }{\left(q^3;q^2\right)_n}=
	(1-q)\sum_{\substack{n\geq 0\\ 0\le j\le n}}
	\frac{ (-1)^n \left(1+q^{2n+2}\right) \left(1+q^{2j+1}\right) q^{3n^2+6n-2j^2-3j} }{1-q^{2n+2}}
.	
\end{align}
Thus, changing $n \mapsto n+1$ in the definition of $U2$, and then using \eqref{Thetid}, we obtain the claim. 
\end{proof}

To relate the parity of $\fthree(n)$ to norms of ideals in $\mathbb Q(\sqrt{6})$, 
we rewrite the sum in Proposition \ref{PropositionF3Mod2}.

\begin{proposition}\label{PropositionF32}
We have that
\begin{align*}
\sum_{\substack{n\ge0\\0\le j\le n}}
\left(1+q^{2j+1}\right) q^{3n^2+6n-2j^2-3j+2} 
&=
	\frac{1}{2}\sum_{\substack{N\ge3\\ N\equiv 2\!\!\!\pmod{4}}}
	\sum_{\substack{-\frac{N}{3}<J\le \frac{N}{3}\\J\equiv 1\!\!\!\pmod{2}}}
	q^{\frac{N^2-6J^2}{16}+\frac18}
.
\end{align*}
\end{proposition}
\begin{proof}
Letting $n\mapsto n+j$, and then swapping $n$ and $j$,
we rewrite the left-hand side as
\begin{equation*}
\sum_{n,j\ge0} q^{n^2+3n+3j^2+6j+6jn+2}
+
\sum_{n,j\ge0} q^{n^2+5n+3j^2+6j+6jn+3} = \sum_{\substack{n\geq 1 \\ -\frac{n}{3}\le j\le\frac{n}{3}}} q^{n^2+n-6j^2-3j}
.
\end{equation*}
For the last step, we let $n\mapsto n-1$ in the first double sum and in the second we let $n\mapsto n+1$.

To finish the claim, we have to prove that
\begin{align}\label{E:F3Parity2}
2\sum_{\substack{ n\geq 1 \\ -\frac{n}{3}\le j\le \frac{n}{3}}} q^{(4n+2)^2-6(4j+1)^2} 
&=\sum_{\substack{N\ge3\\ N\equiv 2\!\!\!\pmod{4}}}
	\sum_{\substack{-\frac{N}{3}<J\le \frac{N}{3}\\J\equiv 1\!\!\!\pmod{2}}}
	q^{N^2-6J^2}
.
\end{align}
Substituting $N=4n+2$ and $J=4j\pm1$, we find that the right-hand side equals
\begin{multline*}
	\sum_{\substack{n\geq 1 \\ -\frac{4n+5}{12}< j\le \frac{4n-1}{12}}}  q^{(4n+2)^2-6(4j+1)^2} 
	+
	\sum_{\substack{ n\geq 1 \\ -\frac{4n-1}{12}< j\le \frac{4n+5}{12}}} q^{(4n+2)^2-6(4j-1)^2} \\
\begin{matrix}
=\displaystyle
2\sum_{n\ge1}
\vphantom{\displaystyle
		\sum_{-\frac{n}{3}\le j\le \frac{n}{3}}
		+\sum_{-\frac{n}{3}-\frac{5}{12}< j< -\frac{n}{3}}
		-\sum_{\frac{n}{3}-\frac{1}{12}< j\le \frac{n}{3}}
}
\end{matrix}
\begin{pmatrix}\displaystyle
		\sum_{-\frac{n}{3}\le j\le \frac{n}{3}}
		+
		\sum_{-\frac{n}{3}-\frac{5}{12}< j< -\frac{n}{3}}
		-
		\sum_{\frac{n}{3}-\frac{1}{12}< j\le \frac{n}{3}} 
\end{pmatrix}
\begin{matrix}
\displaystyle
q^{(4n+2)^2-6(4j+1)^2} 
\vspace{10pt}
\end{matrix}
,
\end{multline*}
letting $j\mapsto -j$ in the second sum of the left-hand side.
From this it is not hard to prove \eqref{E:F3Parity2}.
\end{proof}

We are now ready to prove Theorem \ref{thm:f3odd}.

\begin{proof}[Proof of Theorem \ref{thm:f3odd}]
The proof requires a small amount of standard algebraic
number theory. For the reader not familiar with the definitions, we offer
\cite[Chapter 11]{Baker1} and \cite{Marcus1} as two references.
By Propositions \ref{PropositionF3Mod2} and \ref{PropositionF32},
\begin{align*}
q^{-2}\FThree\left(1;-q^{16}\right)
&\equiv
	\frac{1}{2}\sum_{\substack{N\ge3\\ N\equiv 2\!\!\!\pmod{4}}}
	\sum_{\substack{-\frac{N}{3}<J\le \frac{N}{3}\\J\equiv 1\!\!\!\pmod{2}}}
	q^{N^2-6J^2}
\pmod{2}.
\end{align*}
We note that $N\equiv 2\pmod{4}$ and $J\equiv 1\pmod{2}$ if and only if
$N^2-6J^2\equiv 6\pmod{8}$.

The case $D=6$ of \cite[Lemma 3]{AndrewsDysonHickerson1} 
states that each equivalence class of solutions to
$u^2-6v^2=m$, with $m$ positive, contains a unique $(u,v)$ such that
$u>0$ and $-\frac{u}{3}<v\le \frac{u}{3}$. Recall that two solutions
$(u,v)$ and $(u',v')$ are \textit{equivalent} if
${u'+v'\sqrt{6}}={\pm(5+2\sqrt{6})^r(u+v\sqrt{6})}$ with $r\in\Z$.
As such, two solutions $(u,v)$ and $(u',v')$ are equivalent exactly
if $u+v\sqrt{6}$ and $u'+v'\sqrt{6}$ generate the same ideal in
$\mathcal{O}_K$, where $K:=\mathbb{Q}(\sqrt{6})$. Since $\mathcal{O}_K$ is a principal ideal domain, we see that
\begin{align*}	
\sum_{\substack{N\ge3\\ N\equiv 2\!\!\!\pmod{4}}}
\sum_{\substack{-\frac{N}{3}<J\le \frac{N}{3}\\J\equiv 1\!\!\!\pmod{2}}}
q^{N^2-6J^2}
&=
\sum_{\substack{\mathfrak{a}\subseteq\mathcal{O}_K\\ N(\mathfrak{a})>0\\ N(\mathfrak{a})\equiv6\!\!\!\pmod{8}}}
q^{N(\mathfrak{a})}
.
\end{align*}

Let $a(m)$ denote the number of ideals of $\mathcal{O}_K$ of norm $m$.
A formula for $a(m)$ can be determined by standard methods 
and for our choice of $K$ it is given in the proof of Theorem 1.3 of 
\cite{Lovejoy2}. In particular, suppose that $m$ is positive and 
$m=2^a3^bp_1^{e_1}\dotsm p_j^{e_j}r_1^{f_1}\dotsm r_k^{f_k}s_1^{g_1}\dotsm s_\ell^{g_\ell}$,
where the $p_t$, $r_t$, and $s_t$ are distinct primes with
$p_t\equiv \pm7,\pm11\pmod{24}$,
$r_t\equiv 1,19\pmod{24}$,
and
$s_t\equiv 5,23\pmod{24}$.
Then 
\begin{align*}
a(m)
&=
\begin{cases}
	0 & \mbox{if any $e_t$ is odd or }a+{\displaystyle\sum_{1\le t\le \ell}} g_t\mbox{ is odd},\\[-2ex]
	(f_1+1)\dotsm(f_k+1)(g_1+1)\dotsm(g_\ell+1) & \mbox{otherwise}.
\end{cases}
\end{align*}
It is not hard to see that the parity of $\fthree(n)$ is as claimed,
since $\fthree(n)\equiv \frac{1}{2}Q(16n-2) \pmod{2}$.
\end{proof}

\section{Concluding Remarks}
This paper introduces and proves various properties of the functions 
$\FOne(\zeta;q)$, $\FTwo(\zeta;q)$, and $\FThree(\zeta;q)$.
It is not difficult to see additional results remain and 
so we briefly mention a few of these. The interpretation of $\fone(n)$ and $\fone(m,n)$ is as the
difference of non-negative counts, but it appears that $\fone(m,n)$
is always non-negative. We leave it as an open problem to
prove that $\fone(m,n)$  
is non-negative and to give a combinatorial
interpretation that clearly demonstrates this. Theorem 1.2 gives the asymptotics of $\ftwo(n)$ and $\fthree(n)$.
One could also ask for the asymptotics of 
$\fone(m,n)$, $\ftwo(m,n)$, and $\fthree(m,n)$, as well as how these
ranks are asymptotically distributed as $n\rightarrow \infty$. Also one
could introduce and study the moments of these rank functions.


\begin{thebibliography}{99}

\bibitem{Andrews1} G. Andrews,
{\it The theory of partitions},
Encyclopedia of  Mathematics and its Applications, Vol. 2,
Addison-Wesley Publishing Co., Reading, Mass.-London-Amsterdam, 1976.

\bibitem{Andrews4}
G. Andrews,
{\it Ramanujan's ``lost'' notebook. {IV}. {S}tacks and alternating parity in partitions},
{Adv. in Math.} \textbf{ 53} (1984), 55--74.

\bibitem{Andrews2} G. Andrews,
{\it {$q$}-series: their development and application in analysis,
	number theory, combinatorics, physics, and computer algebra},
volume 66 of CBMS Regional Conference Series in Mathematics,
Published for the Conference Board of the Mathematical,
1986.

\bibitem{AndrewsBerndt1} 
G. Andrews and B. Berndt, 
{\it Ramanujan's lost notebook. {P}art {I}},
{Springer, New York}, 2005.

\bibitem{AndrewsBerndt2} 
G. Andrews and B. Berndt,
{\it Ramanujan's lost notebook. {P}art {II}},
{Springer, New York}, 2009.

\bibitem{AndrewsDysonHickerson1}
G. Andrews, F. Dyson, and D.~Hickerson,
{\it Partitions and indefinite quadratic forms},
Invent. Math. {\bf 91} (1988), 391--407.

\bibitem{AndrewsWarnaar1}
G. Andrews and S. Warnaar,
{\it  The {B}ailey transform and false theta functions},
{Ramanujan J.} \textbf{14} (2007), 173--188.

\bibitem{AtkinSwinnertonDyer1} 
A.~Atkin and P.~Swinnerton-Dyer, 
{\it Some properties of partitions}, Proc. London Math. Soc. {\bf 4} (1954),
84--106.

\bibitem{Auluck1} 
F.~Auluck, 
{\it On some new types of partitions associated with generalized {F}errers graphs}, 
Proc. Cambridge Philos. Soc. \textbf{47} (1951), 84--106.

\bibitem{Baker1} 
A.~Baker,
{\it A comprehensive course in number theory},
Cambridge University Press, Cambridge, 2012.

\bibitem{BarnettFolsomUkoguWesleyXu1}
M.~Barnett, A.~Folsom, O.~Ukogu, W.~Wesley, and H.~Xu, 
{\it Quantum Jacobi forms and balanced unimodal sequences},
J. Number Theory \textbf{186} (2018), 16--34. 

\bibitem{BarnettFolsomWesley1}
M.~Barnett, A.~Folsom, and W.~Wesley, 
{\it Rank generating functions for odd-balanced unimodal sequences, quantum Jacobi forms, and mock Jacobi forms}, 
preprint (2018). 

\bibitem{BerkovichGarvan1}
A.~Berkovich and F.~Garvan,
{\it  Some observations on {D}yson's new symmetries of partitions},
{J. Combin. Theory Ser. A} \textbf{100} (2002), 61--93.

\bibitem{BringmannCreutzigRolen1}  K. Bringmann, T. Creutzig, and L. Rolen, 
{\it Negative index Jacobi forms and quantum modular forms},
Res. Math. Sci. \textbf{1} (2014).

\bibitem{BringmannFolsomOnoRolen1} K. Bringmann, A. Folsom, Ken Ono, and L. Rolen, 
{\it Harmonic Maass forms and mock modular forms: theory and applications}, AMS Colloquium Series \textbf{64},
American Mathematical Society, Providence, RI,  2017.

\bibitem{BringmannJenningsShafferMahlburg1}
K.~Bringmann, C.~Jennings-Shaffer, and K.~Mahlburg,
{\it On a Tauberian theorem of Ingham and Euler-Maclaurin summation}, 
in preparation.

\bibitem{BringmannOno1}
K.~Bringmann and K.~Ono,
{\it Dyson's ranks and {M}aass forms},
{Ann. of Math.}, \textbf{171} (2010), 419--449.

\bibitem{BrysonOnoPitmanRhoades1}
J.~Bryson, K.~Ono, S.~Pitman, and R. Rhoades,
\newblock {\em Unimodal sequences and quantum and mock modular forms.}
\newblock { Proc. Natl. Acad. Sci. USA} \textbf{109} (2012), 16063--16067.

\bibitem{Dyson1} F.~Dyson, 
{\it Some guesses in the theory of partitions},
Eureka (Cambridge) \textbf{8} (1944), 10--15.

\bibitem{Garvan1}
F. Garvan,
{\it  New combinatorial interpretations of {R}amanujan's partition congruences mod {$5,7$} and {$11$}},
{Trans. Amer. Math. Soc.} \textbf{305} (1988), 47--77.

\bibitem{GasperRahman1}
G.~Gasper and M.~Rahman, 
{\em Basic hypergeometric series}, Encyclopedia of  Mathematics and its Applications 96,
Cambridge University Press, Cambridge, second edition, (2004).

\bibitem{HardyRamanujan1} 
G.~Hardy and S.~Ramanujan, 
\emph{Asymptotic Formulaae in Combinatory Analysis}, 
Proc. London Math. Soc. (2) {\bf 17} (1918), 75--115.

\bibitem{KimLimLovejoy1} 
B. Kim, S. Lim, and J. Lovejoy, 
{\it Odd-balanced unimodal sequences and related functions: parity, mock modularity and quantum modularity},
Proc. Amer. Math. Soc. \textbf{144} (2016), 3687--3700.

\bibitem{Lovejoy1}
J.~Lovejoy,
{\it  Lacunary partition functions},
{Math. Res. Lett.} \textbf{9} (2002), 191--198.

\bibitem{Lovejoy2}
J.~Lovejoy,
{\it  Overpartitions and real quadratic fields},
{J. Number Theory} \textbf{106} (2004), 178--186.

\bibitem{Lovejoy3}
J.~Lovejoy,
{\it  Rank and conjugation for the {F}robenius representation of an overpartition},
{Ann. Comb.} \textbf{9} (2005), 321--334.

\bibitem{Lovejoy4}
J.~Lovejoy,
{\it  Rank and conjugation for a second {F}robenius representation of an overpartition},
{Ann. Comb.} \textbf{12} (2008), 101--113.

\bibitem{LovejoyOsburn1}
J.~Lovejoy and R.~Osburn,
{\it  {$M_2$}-rank differences for partitions without repeated odd parts},
{J. Th\'{e}or. Nombres Bordeaux}, \textbf{21} (2009), 313--334.

\bibitem{Marcus1}
D.~A.~Marcus
{\it Number fields},
Springer, Cham,  second edition, (2018).

\bibitem{Mortenson1}
E. Mortenson, 
{\it On the dual nature of partial theta functions and {A}ppell-{L}erch sums}, 
Adv. Math. \textbf{264} (2014), 236--260.

\bibitem{Ramanujan1} 
S.~ Ramanujan, {\it Congruence properties of partitions},
Math. Z. {\bf 9} (1921), 147--153.

\bibitem{Rhoades1}
R. Rhoades,
{\it Asymptotics for the number of strongly unimodal sequences},
{Int. Math. Res. Not.} (2014), 700--719.

\bibitem{Stanley1}
R.~Stanley
{\it Unimodal sequences arising from Lie algebras. Combinatorics, representation theory and statistical methods in groups}, 
Lecture Notes in Pure and Appl. Math. \textbf{57}, 127–-136, Dekker, New York, 1980. 

\bibitem{Stanley2}
R.~Stanley
{\it Log-concave and unimodal sequences in algebra, combinatorics, and geometry}, 
Graph theory and its applications: East and West (Jinan, 1986),  
Ann. New York Acad. Sci. \textbf{576}, 500--535, New York Acad. Sci., New York, 1989. 

\bibitem{Wright1}
E. Wright,
{\it Stacks},
Quart. J. Math. Oxford Ser. \textbf{19} (1968), 313--320.

\bibitem{Zagier1}
D.~Zagier,
\newblock {\em Quantum modular forms,}
\newblock In {Quanta of maths}, volume~11 of {\em Clay Math. Proc.},
659--675. Amer. Math. Soc., Providence, RI, 2010.

\bibitem{Zwegers1} S. Zwegers, 
{\it Mock theta functions}, Ph.D. Thesis, Universiteit Utrecht, 2002.

\bibitem{Zwegers2} S. Zwegers, 
{\it Multivariable {A}ppell functions and nonholomorphic {J}acobi forms}, 
Res. Math. Sci. \textbf{6} (2019).
	
	
\end{thebibliography}
\end{document}